\documentclass[a4paper,12pt,reqno]{amsart}
\usepackage{amsmath, amsthm, amssymb, amstext}

\usepackage[left=2.3cm,right=2.3cm,top=2.5cm,bottom=2.5cm]{geometry}

\usepackage{hyperref,xcolor}
\usepackage{tikz}
\usepackage[hyperpageref]{backref}
\hypersetup{pdfborder={0 0 0},colorlinks}

\usepackage{enumitem}
\allowdisplaybreaks
\newtheorem{theorem}{Theorem}[section]

\newtheorem{remark}{Remark}   
\newtheorem{lemma}[theorem]{Lemma}

\newtheorem{corollary}[theorem]{Corollary}

\numberwithin{theorem}{section}
\numberwithin{equation}{section}

\title[The role of the mean curvature in nonlinear problem]
{The role of the mean curvature in nonlinear
$p$-Laplacian problems with critical exponent}

\author[H. Chtioui]{Hichem Chtioui}
\address[H. Chtioui]{Department of Mathematics, Faculty of Sciences of Sfax, Sfax University, Tunisia}
\email{hichem.chtioui@fss.rnu.tn}

\author[H. Hajaiej]{Hichem Hajaiej}
\address[H. Hajaiej]{Department of Mathematics, California State University,
Los Angeles, CA 90032, USA.}
\email{hichem.hajaiej@gmail.com}

\author[L. Sharma]{Lovelesh Sharma}
\address[L. Sharma]{Department of Mathematics, Indian Institute of Technology Delhi, Hauz Khas, Delhi, India}
\email{slovelesh.96@gmail.com}


\begin{document}

\begin{abstract}
We deal with critical nonlinear problems involving the $p$-Laplacian operator on bounded domains of $\mathbb{R}^n$, $n \geq 2$, with mixed boundary conditions. Using the minimizing technique  {introduced by} Aubin \cite{aubin1976equations} and Br\'ezis-Nirenberg \cite{brezis1983positive}, we prove the existence of least energy solutions. Our work shows a  {significant} difference  between the semilinear case, $p=2$, \cite{adimurthi1991neumann, wang1991neumann} and the quasilinear case, $p\neq 2$ for the existence results.  {Moreover, neither the results for $p=2$ can be extended to $p\neq2$, nor our findings for $p\neq2$ can apply to $p=2$. Additionally, the cases $(p<2 ~\text{and}~ p>2)$ present different challenges and need to be studied separately. More precisely, when $p>2$, the effect of the geometry of the boundary conditions dominates that one of the potential, whereas for $p<2$ the opposite behavior holds true.}
\end{abstract}

\maketitle

\section{Introduction}

Let $n \geq 2$ and $\Omega$ be a bounded domain of $\mathbb{R}^n$ such that $\partial\Omega$ is Lipschitz-continuous and decomposed into two disjoint smooth manifolds $\Gamma_0$ and $\Gamma_1$. Suppose that the $(n-1)$-dimensional Hausdorff measure of $\Gamma_1$ is positive. In this paper we are interested in the existence of positive solutions for the following critical nonlinear  {critical} problem:
\begin{equation}\label{eq1.1}
\begin{aligned}
\begin{cases}
-\Delta_p u + \alpha(x)|u|^{p-2}u
&=~~ |u|^{p^*-2}u \quad \text{in } \Omega,\\[1mm]
~~~~\quad u
&=~~ 0 \quad \text{on } \Gamma_0,\\[1mm]
|\nabla u|^{p-2}\dfrac{\partial u}{\partial \nu} + \beta(x)|u|^{p-2}u
&=~~ 0 \quad \text{on } \Gamma_1,
\end{cases}
\end{aligned}
\tag{P}
\end{equation}
where $\Delta_p$, $1<p<n$, is the $p$- {Laplacian} operator defined on $W^{1,p}(\Omega)$ as
\[
\Delta_p u = \operatorname{div}\bigl(|\nabla u|^{p-2}\nabla u\bigr),
\]
and
\(
p^* = \frac{np}{n-p}
\)
is the critical Sobolev exponent, $\nu$ is the unit exterior normal to $\Gamma_1$ and $\alpha(x)\in L^\infty(\Omega)$, $\beta(x)\in L^\infty(\Gamma_1)$ are two smooth functions such that
\begin{equation}
\label{eq1.2}
\int_\Omega \Bigl(|\nabla u|^p + \alpha(x)|u|^p\Bigr)\,dx
+ \int_{\Gamma_1} \beta(x)|u|^p\,d\sigma
\geq c \int_\Omega |u|^p\,dx,
\end{equation}
for any $u\in W^{1,p}(\Omega)$ with $c>0$. Define
\[
V^{1,p}(\Omega)=\bigl\{u\in W^{1,p}(\Omega)\,;\, u=0 \text{ on } \Gamma_0\bigr\}.
\]
Of course if $\Gamma_0=\varnothing$, $V^{1,p}(\Omega)=W^{1,p}(\Omega)$. For $u\in V^{1,p}(\Omega)$, we define
\[
\|u\|^p
=
\int_\Omega \Bigl(|\nabla u|^p + \alpha(x)|u|^p\Bigr)\,dx
+
\int_{\Gamma_1}\beta(x)|u|^p\,d\sigma.
\]
Under inequality \eqref{eq1.2}, $\|\cdot\|$ defines a norm on $V^{1,p}(\Omega)$ which is equivalent to the usual norm of $W^{1,p}(\Omega)$. Problem \eqref{eq1.1} has a variational structure. If $u$ is a weak solution of \eqref{eq1.1} in the sense that
\[
\int_\Omega |\nabla u|^{p-2}\nabla u \cdot \nabla h\,dx
+
\int_\Omega \alpha(x)|u|^{p-2}uh\,dx
+
\int_{\Gamma_1}\beta(x)|u|^{p-2}uh\,d\sigma
=
\int_\Omega |u|^{p^*-2}uh\,dx,
\]
for all \( h\in V^{1,p}(\Omega)\),
then $u$ is a critical point (up to a positive multiplicative constant) of the energy functional
\begin{equation}
\label{eq1.3}
J(u)
=
\frac{\|u\|^p}
{\left(\displaystyle\int_\Omega |u|^{p^*}\,dx\right)^{\frac{p}{p^*}}},
\qquad
u\in V^{1,p}(\Omega)\setminus\{0\}.
\end{equation}
The Sobolev quotient $Q_p(\Omega)$ is defined by
\begin{equation}
\label{eq1.4}
Q_p(\Omega)=\inf_{\substack{u\in V^{1,p}(\Omega)\setminus\{0\}}} J(u).
\end{equation}
From the Sobolev inequality of the critical embedding
\[
V^{1,p}(\Omega)\hookrightarrow L^{p^*}(\Omega),
\]
see \cite{maz2013sobolev}, and from \eqref{eq1.2}, we can deduce that $Q_p(\Omega)>0$.  {Unlike the pure homogeneous Dirichlet case where $Q_p(\Omega)$ depends only on the domain $\Omega$, it was proved in \cite{lions1988best} that under the mixed Dirichlet-Neumann boundary condition, $Q_p(\Omega)$
depends on both $\Omega$ and $\Gamma_1$}. Moreover, it is proved in \cite[Corollary 2.2]{lions1988best} that for $\alpha(x)=0$ on $\Omega$, $\beta(x)=0$ on $\Gamma_1$ and $\Gamma_0\neq \varnothing$, the Sobolev quotient $Q_p(\Omega)$ can be achieved if $\Omega$ belongs to a class of bounded domains defined according to some geometric property of $\Gamma_1$.

The critical $p$-Laplacian problems have been the subject of  {many studies, when $p=2$, the equation reads as:}

\begin{equation}
\label{eq1.5}
-\Delta u+\alpha(x)u=|u|^{2^*-2}u.
\end{equation}
Numerous studies with important results have been obtained on problem \eqref{eq1.5} under various boundary conditions. We refer  {the readers}, for example to \cite{bahri1987nonlinear,brezis1983positive,coron1984topologie} for the Dirichlet boundary condition, \cite{yadava1993conjecture, wang1991neumann, adimurthi1991existence} for the Neumann boundary condition and \cite{adimurthi1991neumann, maz2013sobolev, alghanemi2023towards} for mixed boundary conditions. On closed Riemannian manifolds, problem \eqref{eq1.5} is related to the Yamabe  {problem} or, more generally,  {to} the scalar curvature problem. For this topic, we refer the reader to the works of \cite{aubin1976equations, bahri1991scalar, chtioui2023chen, schoen1984conformal, li1995prescribing} and the references therein.

Over the past decades, considerable efforts have been made to extend studies on problem \eqref{eq1.5} to the quasilinear case, $p\neq 2$.  {However, the main focus was on} the $p$-Laplacian problems under the pure Dirichlet boundary conditions, see for example  {\cite{anello2002infinitely, bonanno2003multiplicity, carlos2026ground}} or under the pure Neumann boundary conditions, see for example  {\cite{deng2025existence, huang1990eigenvalue}}, and \cite{wang1993positive}. In contrast to this,  {a} very few papers are known for $p\neq 2$ under mixed boundary conditions. In this direction we refer to the aforementioned paper \cite{lions1988best}.

In \cite{adimurthi1991neumann}, Adimurthi- Mancini considered problem \eqref{eq1.1} for $p=2$. Under a suitable geometrical condition on $\Gamma_1$,  {they} were able to establish existence results in the case  $\beta(x)=0$ on $\Gamma_1$ (see \cite[Theorems 1.1 and 1.2]{adimurthi1991neumann}).  {They also addressed the more general case $\beta(x)\neq 0$, but under some restrictive conditions on $\beta(x)$}, (see \cite[Theorem 3.1]{adimurthi1991neumann}). More precisely, let $H(x_0)$, $x_0\in \Gamma_1$, be the mean curvature with respect to the unit exterior normal at $x_0$ and denote
\[
\beta^+(x)=\max\bigl(\beta(x),0\bigr).
\]
 {Their, result is the follows:}
\begin{theorem}\label{thm1.1}
\cite{adimurthi1991neumann}
Let $p=2$, $n\geq 3$, and let $\alpha(x)$ and $\beta(x)$ be two functions satisfying condition \eqref{eq1.2}.
Assume that the following conditions hold:

\medskip

\noindent\textbf{(g.c.)}
There exists $x_0$ in the interior of $\Gamma_1$ such that
\[
H(x_0) > 0,
\]
and, in a neighborhood of $x_0$, $\Omega$ lies on one side of the tangent
space of $\Gamma_1$ at $x_0$.

\medskip

\noindent\textbf{($\beta$.c.)}
The function $\beta(x)$ satisfies one of the following conditions:
\begin{enumerate}
\item $\|\beta^+\|_{L^\infty(\Gamma_1)} < \dfrac{n-2}{2} H(x_0)$,
\item $\beta(x)=O(|x-x_0|^k)$ for $x$ close to $x_0$ and $k>0$.
\end{enumerate}
Then problem \eqref{eq1.1} admits a positive solution $u$ such that
\[
J(u)=Q_p(\Omega).
\]
\end{theorem}
From the above results, we observe that the existence of positive solutions for problem \eqref{eq1.1} when $p=2$ is based on both conditions; \textbf{(g.c.)} on $\Gamma_1$ and \textbf{($\beta$.c.)} on the potential function $\beta(x)$. Actually if one of these two conditions is removed, the result of Theorem \ref{thm1.1} does not hold  {true}. Indeed, for $\Omega$ an open part of $\mathbb{R}^n$ bounded by two concentric spheres with $\Gamma_1$ describes the interior sphere, \textbf{(g.c.)}-condition is not satisfied and for $\beta(x)=0$, \textbf{($\beta$.c.)}-condition is satisfied. It is proved in \cite{PacellaTricarico} by using certain isoperimetric arguments that the Sobolev quotient $Q_p(\Omega)$ is is not achieved  {regardless of the radius} of the two spheres.

The same observation can be made on the work of Wang \cite{wang1991neumann}, where
problem \eqref{eq1.1} was studied for $p=2$ under the condition $\Gamma_0=\varnothing$ (in this case, the above assumption \textbf{(g.c.)} is satisfied since $\Gamma_1=\partial\Omega$) and other
conditions on $\beta(x)$,  {(see, Corollaries 2.1 and 2.2 of \cite{wang1991neumann}, for more details)}. See also the paper \cite{grossi1990positive} where the
authors proved the existence of positive solutions for the problem \eqref{eq1.1}
when $p=2$, $\beta(x)=0$ and some conditions on $\Omega, \Gamma_0$, and   $\Gamma_1$.

For $p\neq2$,  {the study is more subtle and delicate}. Indeed, we shall prove in this
paper that the contribution of the mean curvature $H(x)$ of the boundary part
$\Gamma_1$ in the variational analysis associated to problem \eqref{eq1.1}
is of order
\[
\frac{H(x)}{\lambda^{p-1}}
\]
where $\lambda$ is a large parameter, while the contribution of the potential
function $\beta(x)$ is of order
\[
\frac{\beta(x)}{\lambda^{(p-1)^2}}.
\]
It follows that for $p>2$ the effect of the boundary geometry
dominates the effect of the potential $\beta(x)$.
For $p<2$ the reverse happens, while for $p=2$ there is a balance
 {between the two effects}. This leads to two different kinds of existence results
when $p\neq2$.

In the first result of this paper, we do not assume any geometrical condition on $\Gamma_1$. Namely.

\begin{theorem}\label{thm1.2}
Let $1<p<2$, $n\ge3$ and $\alpha(x)$ and $\beta(x)$ be two functions
satisfying condition \eqref{eq1.2}.
If $\beta(x)$ is negative somewhere on $\Gamma_1$, then problem \eqref{eq1.1}
admits a positive solution $u$ such that
\[
J(u)=Q_p(\Omega).
\]
    \end{theorem}
\begin{remark}

Note that Theorem \ref{thm1.2}  {can not hold true for $p=2$.}
Indeed,  {when} $u$ is a solution of problem \eqref{eq1.2}, with $p=2$, it is not difficult to see that the
Pohozaev's identity becomes

\[
\frac1{2^*} \int_{\Gamma_1} u^{2^*} (x\cdot \nu) d\sigma
=
\frac12 \int_{\Gamma_1} \alpha(x)u^2 (x\cdot \nu)d\sigma
+
\frac12 \int_{\Gamma_1} |\nabla u|^2 (x\cdot \nu)d\sigma
\]
    \end{remark}
\[
- \int_{\Gamma_1} (\beta(x) u)^2 (x\cdot \nu)d\sigma-
\frac12 \int_{\Gamma_0}
\left(\frac{\partial u}{\partial \nu}\right)^2
(x\cdot\nu)d\sigma
-
\int_{\Omega} \alpha(x)u^2 dx
+
\frac{n-2}{2}\int_{\Gamma_1}\beta(x)u^2 d\sigma.
\]
Let $\Omega$ be the intersection of a smooth cone $\mathcal{C}$ with the vertex at $0_{\mathbb{R}^n}$
with the ball $B(0_{\mathbb{R}^n},1)$.
Suppose that
\[
\Gamma_0 = \mathcal{C} \cap \partial B(0_{\mathbb{R}^n},1),
\quad
\Gamma_1 = \partial\Omega \setminus \Gamma_0.
\]
Thus we have \[x\cdot\nu = 0 \ on \ \Gamma_1 \  and \
x\cdot\nu > 0 \  on \  \Gamma_0.\]
It follows from the above identity that
\[
\frac{n-2}{2}\int_{\Gamma_1}\beta(x)u^2 d\sigma
=
\frac12 \int_{\Gamma_0}
\left(\frac{\partial u}{\partial \nu}\right)^2
(x\cdot\nu)d\sigma
+
\int_{\Omega}\alpha(x)u^2 dx .
\]
 {Therefore, if $\beta(x)$ is non-positive function and $\alpha(x)$ is positive function, then problem \eqref{eq1.1} does not admit positive solution.}

\begin{remark}
     {In} $n=2$, Theorem \ref{thm1.2} holds for $1<p\leq \frac{3}{2}$, see Lemma \ref{lem2.7}.
\end{remark}
In the second result of this paper, we assume the geometrical condition \textbf{(g.c)} and no assumption on $\beta(x)$, except the coercivity condition \eqref{eq1.2}.

\begin{theorem}\label{thm1.3}
Let
\(
2<p\le\frac{n+1}{2}, n\geq 3,
\)
and $\alpha(x)$, $\beta(x)$ satisfy condition \eqref{eq1.2}.
Let $\Omega$ be a bounded domain of $\mathbb{R}^n$
satisfying condition \textbf{(g.c.)}.
Then problem \eqref{eq1.1} has a positive solution $u$ such that
\[
J(u)=Q_p(\Omega).
\]
    \end{theorem}
\begin{remark}
The restriction  on $p$ in the above theorem,
\(p\le\frac{n+1}{2}
\)
is due to a technical reason related to the convergence
of some integrals. However, under some  {additional} conditions on
$\beta(x)$, the condition $p$ may be relaxed.
    \end{remark}
More precisely the following holds.

\begin{theorem}\label{thm1.4}
Let
\(
1<p\le\frac{n+1}{2}
\), \(n\geq 2\).
Assume that $\beta(x)=0$ on $\Gamma_1$ and
$\alpha(x)\in L^\infty(\Omega)$ satisfies condition \eqref{eq1.2}.
If condition \textbf{(g.c.)} holds, then problem \eqref{eq1.1} has a positive
solution minimizing the energy functional $J$.

\end{theorem}

In the next section, we prove our existence results.
We follow the minimizing argument first introduced by
Aubin  \cite{aubin1976equations} and  {later} developed by Br\'ezis-Nirenberg \cite{brezis1983positive} for
semilinear critical problems with Dirichlet boundary conditions.
The method consists in proving that the Sobolev quotient
$Q_p(\Omega)$ defined in \eqref{eq1.4} is below the first level
at which the Palais-Smale condition is not satisfied.
Consequently any minimizing sequence of the energy
functional $J$ satisfies the Palais-Smale condition
and hence converges (up to a subsequence) to a minimizing
function.
\section{Proof of the existence results}

We begin by recalling the Sobolev constant
\[
S=\inf_{\substack{u\in W^{1,p}_0(\Omega)\\ u\ne 0}}
\frac{\displaystyle \int_\Omega |\nabla u|^p\,dx}
{\left(\displaystyle \int_\Omega |u|^{p^*}\,dx\right)^{p/p^*}},
\]
where
\[
W^{1,p}_0(\Omega)
=
\left\{
u\in W^{1,p}(\Omega):\; u=0 \ \text{on } \partial\Omega
\right\}.
\]
It is proved in \cite{talenti1976best} that $S$ is independent of the domain $\Omega$
and it is never achieved except when
\(
\Omega=\mathbb{R}^n,
\)
and $W^{1,p}_0(\Omega)$ is replaced by
\[\{u\in L^{p^*}{(\Omega)}, ~\frac{\partial u}{\partial x_i}\in L^p(\Omega),~i=1,2,\dots,n \}.
\]
In this case, the unique minimizers of $S$ are the functions(called the
Aubin-Talenti bubbles) of the form
\begin{equation}\label{eq2.0}
\delta_{a,\lambda}(x)
=
\left(
\frac{\lambda^{p-1}}
{1+\lambda^{p}|x-a|^{\frac{p}{p-1}}}
\right)^{\frac{n-p}{p}}, ~~x\in \mathbb{R}^n,
\end{equation}
where $a\in \mathbb{R}^n$ and $\lambda>0.$

Let $Q_p(\Omega)$ be the Sobolev quotient defined in \eqref{eq1.4}.
Following \cite[Lemma 2.1]{adimurthi1991neumann} and \cite[Corollary 2.1]{lions1988best}, we have the following result.

\begin{lemma}\label{lem2.1}
    $Q_p(\Omega)$ is achieved provided that
\begin{equation}\label{eq2.1}
Q_p(\Omega) < \frac{S}{2^{p/n}}.
\tag{2.1}
\end{equation}
\end{lemma}
In the following, we shall prove inequality \eqref{eq2.1} under the assumptions of each of
our theorems. In order do this, we need to exhibit functions
$u \in V^{1,p}(\Omega)$ which are supported near the boundary
$\Gamma_1$ with

\[
J(u) < \frac{S}{2^{p/n}}.
\]
Let $a \in \Gamma_1$. In a generic case, we may assume that in a small
neighbourhood of $a$, $\Omega$ lies on one side of the tangent space
of $\Gamma_1$ at $a$.
Let $\lambda$ be a large positive constant and we define
\begin{equation}\label{eq2.2}
U_{a,\lambda}(x) = \Psi(x) \delta_{a,\lambda}(x), \qquad x\in \Omega,
\end{equation}
where $\Psi(x)$ is a cut-off function defined in $\mathbb{R}^n$ such that
\[
\Psi(x) =
\begin{cases}
1, & \text{if } |x| < \dfrac{r}{2}, \\[6pt]
0, & \text{if } |x| > r,
\end{cases}
\]
where $r>0$ is a sufficiently small constant.

We now prove the following Lemmas, which gives useful elementary
estimates for the Aubin-Talenti bubbles. These estimates  {are interesting in themselves and can be used for} further critical problems involving the p-Laplacian operator. Let $H(a)$ be the mean curvature of $\Gamma_1$ at $a.$
\begin{lemma}\label{lem2.2}
Let $1<p\le \frac{n+1}{2}$. Then
\[
\int_\Omega |\nabla U_{a,\lambda}|^p \, dx
=
\left(\frac{n-p}{p-1}\right)^p
\begin{cases}
\displaystyle
\Sigma - \frac{(c_1-c_2)H(a)}{\lambda^{p-1}}
+ o\!\left(\frac{1}{\lambda^{p-1}}\right),
& \text{if } p<\frac{n+1}{2}, \\[0.4cm]

\displaystyle
\Sigma - \frac{\hat c H(a)\log \lambda}{\lambda^{p-1}}
+ o\!\left(\frac{\log \lambda}{\lambda^{p-1}}\right),
& \text{if } p=\frac{n+1}{2}.
\end{cases}
\]
Here
\[
\Sigma =
\int_{\mathbb{R}^{n-1}}
\frac{|z|^{\frac{p}{p-1}}}{\left(1+|z|^{\frac{p}{p-1}}\right)^n}
\, dz,
\]

\[
c_1 =
\int_{\mathbb{R}^{n-1}}
\frac{|z|^{2}}{\left(1+|z|^{\frac{p}{p-1}}\right)^{n-1}}
\, dz,
\qquad
c_2 =
\int_{\mathbb{R}^{n-1}}
\frac{|z|^{2}}{\left(1+|z|^{\frac{p}{p-1}}\right)^n}
\, dz,
\]
and $\hat c>0$ is a constant.
\end{lemma}
\begin{proof}
Without loss of generality we may assume that $a=0$.
Let $x'\in\mathbb R^{n-1}\mapsto \varphi(x')$ be the local
parametrization of $\Gamma_1$ near $0$. Therefore in the ball
$B(0,r)$ of centre $0$ and radius $r$, $(0<r<<1)$, we have
\[
B(0,r)\cap\Gamma_1=\{(x',x_n)\in B(0,r):x_n=\varphi(x')\},
\]
\[
B(0,r)\cap\Omega=\{(x',x_n)\in B(0,r):x_n>\varphi(x')\}.
\]
By Taylor expansion of $\varphi(x')$ around $0$, it holds
\begin{equation}\label{eq.2.3}
\varphi(x')=\sum_{i=1}^{n-1}\gamma_i x_i^2+O(|x'|^3),
    \end{equation}
up to some change of coordinates.  According to \eqref{eq.2.3}, we have
\[
H(0)=\frac{2}{n-1}\sum_{i=1}^{n-1}\gamma_i .
\]
Observe that
\begin{equation}\label{eq2.4}
\int_\Omega |\nabla U_{(0,\lambda)}|^p dx
=
\int_{\Omega\cap B(0,\frac{r}{2})} |\nabla \delta_{(0,\lambda)}|^p dx
+
\int_{\Omega\setminus B(0,\frac{r}{2})} |\nabla U_{(0,\lambda)}|^p dx
=: I_0+R_0 .
    \end{equation}
In order to estimate $I_0$, we decompose $B(0,\frac{r}{2})\cap \Omega$ as follows:
\[
B(0,\tfrac{r}{2})\cap\Omega
=\Sigma_1\cup (B^+(0,\tfrac{r}{2})\setminus\Sigma_2),
\]
where
\[
B^+(0,\tfrac{r}{2})
=
\{(x',x_n)\in B(0,\tfrac{r}{2}) : x_n>0\},
\]
\[
\Sigma_1=\{(x',x_n)\in B(0,\tfrac{r}{2}) :\varphi(x')<x_n\leq 0\},\qquad
\Sigma_2=\{(x',x_n)\in B(0,\tfrac{r}{2}) :0\le x_n\le \varphi(x')\}.
\]
Therefore, we write
\begin{equation}\label{eq2.5}
I_0=\int_{B^+(0,\frac{r}{2})} |\nabla \delta_{(0,\lambda)}|^p dx
-\int_{\Sigma_2} |\nabla \delta_{(0,\lambda)}|^p dx
+\int_{\Sigma_1} |\nabla \delta_{(0,\lambda)}|^p dx
=:I_1-I_2+I_3 .
    \end{equation}
By direct computations, we have from \eqref{eq2.0}
\begin{equation}\label{eq2.6}
|\nabla \delta_{(0,\lambda)}|^p
=
\left(\frac{n-p}{p-1}\right)^p
\frac{\lambda^{n(p-1)+p}|x|^{\frac{p}{p-1}}}
{\big(1+\lambda^p |x|^{\frac{p}{p-1}}\big)^n}.
    \end{equation}
Therefore, by setting
\[
z= {\lambda^{{p-1}}}x,
\]
\begin{equation}\label{eq2.7}
\begin{aligned}
I_1
&=
\left(\frac{n-p}{p-1}\right)^p
\left(
\int_{\mathbb{R}^{n}_{+}}
\frac{|z|^{\frac{p}{p-1}}}{\left(1+|z|^{\frac{p}{p-1}}\right)^n}\,dz
+
o\!\left(\frac{1}{\lambda^{\,n-p}}\right)
\right) \\
&=
\left(\frac{n-p}{p-1}\right)^p\Sigma +
\begin{cases}
 o\!\left(\dfrac{1}{\lambda^{\,p-1}}\right), & \text{if } 1<p<\dfrac{n+1}{2},\\[1.2ex]
 o\!\left(\dfrac{\log \lambda}{\lambda^{\,p-1}}\right), & \text{if } p=\dfrac{n+1}{2}.
\end{cases}
\end{aligned}
\end{equation}
Next, in order to estimate $I_2$, we define for $\delta>0$ small enough,
\[
L_\delta=\{(x',x_n)\in\mathbb R^n:|x'|<\delta\}.
\]
Then
\begin{equation*}
\begin{aligned}
I_2
&=
\int_{\Sigma_2\cap L_{\delta}} |\nabla \delta_{(0,\lambda)}|^p \, dx
+
O\!\left(
\int_{L_\delta^c} |\nabla \delta_{(0,\lambda)}|^p \, dx
\right) \\
&=
\left(\frac{n-p}{p-1}\right)^p
\lambda^{n(p-1)}
\int_{\Sigma_2\cap L_{\delta}}
\frac{\lambda^p|x|^{\frac{p}{p-1}}}
{\left(1+\lambda^p |x|^{\frac{p}{p-1}}\right)^n}
\,dx
+
O\!\left(\frac{1}{\lambda^{\,n-p}}\right).
\end{aligned}
\end{equation*}
Observe that
\[
\begin{aligned}
\qquad
\frac{\lambda^p |x|^{\frac{p}{p-1}}}
{\left(1+\lambda^p|x|^{\frac{p}{p-1}}\right)^n}
=
\frac{1}{\left(1+\lambda^p|x|^{\frac{p}{p-1}}\right)^{\,n-1}}
-
\frac{1}{\left(1+\lambda^p|x|^{\frac{p}{p-1}}\right)^n}.
\end{aligned}
\]
It follows that
\begin{equation}\label{eq2.8}
\begin{aligned}
I_2
&=
\left(\frac{n-p}{p-1}\right)^p
\lambda^{n(p-1)}
\Bigg(
\int_{\Sigma_2 \cap L_\delta}
\frac{dx}{\left(1+\lambda^p|x|^{\frac{p}{p-1}}\right)^{n-1}}
-
\int_{\Sigma_2 \cap L_\delta}
\frac{dx}{\left(1+\lambda^p|x|^{\frac{p}{p-1}}\right)^n}
\Bigg)
+
o\!\left(\frac{1}{\lambda^{p-1}}\right) \\
&=:
\left(\frac{n-p}{p-1}\right)^p
\lambda^{n(p-1)}
\left(J_1 - J_2\right)
+
o\!\left(\frac{1}{\lambda^{p-1}}\right).
\end{aligned}
\end{equation}
\noindent
 {Using the following estimate: for any $X,Y\in\mathbb{R}^n$ and $q>1$, we have
\begin{equation*}
|X+Y|^{q}
=
|X|^{q}
+
O\!\left(
|Y|^q
+
|X|^{q-\theta}|Y|^{\theta}
+
|X|^{\theta}|Y|^{q-\theta}
\right),
\end{equation*}
where $\theta>0$ is sufficiently small. Here, $O(\cdot)$ denotes a quantity
which is bounded by a constant multiple of its argument, that is, there exists
a constant $C>0$, independent of $X$ and $Y$, such that
\[
\big||X+Y|^q - |X|^q\big|
\le
C\left(
|Y|^q
+
|X|^{q-\theta}|Y|^{\theta}
+
|X|^{\theta}|Y|^{q-\theta}
\right).
\]}
We write
\begin{equation}\label{eq.2.9}
|x|^{\frac{p}{p-1}}
=
|(x',0)+x_n e_n|^{\frac{p}{p-1}}
= |x'|^{\frac{p}{p-1}}+
O\!\left(
|x_n|^{\frac{p}{p-1}}+|x_n|^{\theta}|x'|^{\frac{p}{p-1}- \theta} +|x'|^{\theta}|x_n|^{\frac{p}{p-1}- \theta}
\right).
    \end{equation}
    Therefore, $J_1$ can be expanded as follows:
   \begin{equation}\label{eq2.10}
\begin{aligned}
J_1
&=
\int_{\Sigma_2\cap L_\delta}
\frac{1}{\left(1+\lambda^p|x'|^{\frac{p}{p-1}}\right)^{n-1}}
\Bigg[
1
+O\!\left(
\frac{\lambda^p|x_n|^{\frac{p}{p-1}}}
     {1+\lambda^p|x'|^{\frac{p}{p-1}}}
\right)
\\
&\qquad
+O\!\left(
\frac{\lambda^p|x'|^{\frac{p}{p-1}-\theta}|x_n|^{\theta}}
     {1+\lambda^p|x'|^{\frac{p}{p-1}}}
\right)
+O\!\left(
\frac{\lambda^p|x_n|^{\frac{p}{p-1}-\theta}|x'|^{\theta}}
     {1+\lambda^p|x'|^{\frac{p}{p-1}}}
\right)
\Bigg]^{-(n-1)} dx
\\[6pt]
&=
\int_{\Sigma_2\cap L_\delta}
\frac{dx}{\left(1+\lambda^p|x'|^{\frac{p}{p-1}}\right)^{n-1}}
+O\!\left(
\int_{\Sigma_2\cap L_\delta}
\frac{\lambda^p|x_n|^{\frac{p}{p-1}}}
     {\left(1+\lambda^p|x'|^{\frac{p}{p-1}}\right)^n}\,dx
\right)
\\
&\quad
+O\!\left(
\int_{\Sigma_2\cap L_\delta}
\frac{\lambda^p|x'|^{\frac{p}{p-1}-\theta}|x_n|^{\theta}}
     {\left(1+\lambda^p|x'|^{\frac{p}{p-1}}\right)^n}\,dx
\right)
\\
&\quad
+O\!\left(
\int_{\Sigma_2\cap L_\delta}
\frac{\lambda^p|x_n|^{\frac{p}{p-1}-\theta}|x'|^{\theta}}
     {\left(1+\lambda^p|x'|^{\frac{p}{p-1}}\right)^n}\,dx
\right)
\\[4pt]
&=: K_1 + R_1 + R_2 + R_3 .
\end{aligned}
\end{equation}
By Fubini's theorem and \eqref{eq.2.3}, we have
\begin{align*}
K_1
&=
\int_{|x'|<\delta}\int_{0}^{\varphi(x')}
\frac{dx_n\,dx'}
{\left(1+\lambda^p|x'|^{\frac{p}{p-1}}\right)^{n-1}}
=
\int_{|x'|<\delta}
\frac{\varphi(x')\,dx'}
{\left(1+\lambda^p|x'|^{\frac{p}{p-1}}\right)^{n-1}}
\nonumber\\
&=
\sum_{i=1}^{n-1}\gamma_i
\int_{|x'|<\delta}
\frac{x_i^2}
{\left(1+\lambda^p|x'|^{\frac{p}{p-1}}\right)^{n-1}}\,dx'
+o\!\left(
\int_{|x'|<\delta}
\frac{|x'|^2}
{\left(1+\lambda^p|x'|^{\frac{p}{p-1}}\right)^{n-1}}\,dx'
\right),
\end{align*}
as \(\delta\) is small enough. Setting \(z=\lambda^{p-1}x'\), we get
\begin{equation*}
\begin{aligned}
K_1
&=
\sum_{i=1}^{n-1}
\frac{\gamma_i}{\lambda^{(p-1)(n+1)}}
\int_{|z|<\lambda\delta}
\frac{z_i^2}
{\left(1+|z|^{\frac{p}{p-1}}\right)^{n-1}}\,dz
+o\!\left(
\frac{1}{\lambda^{(p-1)(n+1)}}
\int_{|z|<\lambda \delta}
\frac{|z|^2\,dz}
{\left(1+|z|^{\frac{p}{p-1}}\right)^{n-1}}
\right).
\nonumber
\end{aligned}
\end{equation*}
If \(1<p<\frac{n+1}{2}\), then
\begin{equation}\label{eq2.11}
\begin{aligned}
K_1
&=
\frac{1}{(n-1)\lambda^{(p-1)(n+1)}}
\sum_{i=1}^{n-1} \gamma_i
\int_{\mathbb{R}^{n-1}}
\frac{|z|^2}
{\left(1+|z|^{\frac{p}{p-1}}\right)^{n-1}}\,dz
+o\!\left(\frac{1}{\lambda^{(p-1)(n+1)}}\right) \\
&=
\frac{c_1}{(n-1)\lambda^{(p-1)(n+1)}}
\sum_{i=1}^{n-1}\gamma_i
+o\!\left(\frac{1}{\lambda^{(p-1)(n+1)}}\right).
\end{aligned}
\end{equation}
If \(p=\frac{n+1}{2}\), then
\begin{equation}\label{eq2.111}
K_1
=
\frac{\widehat c}{n-1}\,
\frac{\left(\sum_{i=1}^{n-1}\gamma_i\right)\log \lambda}
{\lambda^{(p-1)(n+1)}}
+o\!\left(
\frac{\log \lambda}{\lambda^{(p-1)(n+1)}}
\right),
\end{equation}
where \(\widehat c\) is a positive constant.
The remainder terms of \eqref{eq2.10} can be computed as follows:
\[
\begin{aligned}
R_1
&=O\!\left(
\int_{|x'|<\delta}
\left(
\int_{0}^{\varphi(x')}
\frac{\lambda^p|x_n|^{\frac{p}{p-1}}\,dx_n}
{\left(1+\lambda^p|x'|^{\frac{p}{p-1}}\right)^n}
\right)dx'
\right)  \notag \\
&=O\!\left(
\int_{|x'|<\delta}
\frac{\lambda^p(\varphi(x'))^{\frac{p}{p+1}+1}}
{\left(1+\lambda^p|x'|^{\frac{p}{p-1}}\right)^n}
\,dx'
\right) \notag \\
&=O\!\left(
\int_{|x'|<\delta}
\frac{\lambda^p|x'|^{\frac{2p}{p-1}+2}}
{\left(1+\lambda^p|x'|^{\frac{p}{p-1}}\right)^n}
\,dx'
\right).
\end{aligned}
\]
Observe that
\[
\frac{\lambda^p|x'|^{\frac{2p}{p-1}+2}}
{\left(1+\lambda^p|x'|^{\frac{p}{p-1}}\right)^n}
=
o\!\left(
\frac{|x'|^2}
{\left(1+\lambda^p|x'|^{\frac{p}{p-1}}\right)^{n-1}}
\right),~~~~~~\forall~~|x'|<\delta,
\]
as $\delta$ is small enough. Indeed,
\[
\frac{\lambda^p |x'|^{\frac{2p}{p-1}+2}}
{\left(1+\lambda^p |x'|^{\frac{p}{p-1}}\right)^n}
\cdot
\frac{\left(1+\lambda^p |x'|^{\frac{p}{p-1}}\right)^{n-1}}
{|x'|^2}
=
\frac{\lambda^p |x'|^{\frac{2p}{p-1}}}
{1+\lambda^p |x'|^{\frac{p}{p-1}}}.
\]
Since
\[
\frac{\lambda^p |x'|^{\frac{2p}{p-1}}}
{1+\lambda^p |x'|^{\frac{p}{p-1}}}
\le |x'|^{\frac{p}{p-1}},
\]
we obtain
\[
\frac{\lambda^p |x'|^{\frac{2p}{p-1}+2}}
{\left(1+\lambda^p |x'|^{\frac{p}{p-1}}\right)^n}
\cdot
\frac{\left(1+\lambda^p |x'|^{\frac{p}{p-1}}\right)^{n-1}}
{|x'|^2}
\le |x'|^{\frac{p}{p-1}}
\leq O\!\left(\delta^{\frac{p}{p-1}}\right)
\to 0
\quad \text{as } \delta \to 0.
\]
Therefore, as in \eqref{eq2.11} and \eqref{eq2.111}, we have
\begin{equation}\label{eq2.12}
R_1
=
\begin{cases}
\displaystyle
o\!\left(\dfrac{1}{\lambda^{(p-1)(n+1)}}\right),
& \text{if } p<\dfrac{n+1}{2}, \\[8pt]
\displaystyle
o\!\left(\dfrac{\log \lambda}{\lambda^{(p-1)(n+1)}}\right),
& \text{if } p=\dfrac{n+1}{2}.
\end{cases}
\end{equation}
In the same way we have
\begin{align*}
R_2
&=O\!\left(
\int_{|x'|<\delta}
\frac{\lambda^p|x'|^{\frac{p}{p-1}+\theta+2}}
{\left(1+\lambda^p|x'|^{\frac{p}{p-1}}\right)^n}
\,dx'
\right) \notag\\
&=o\!\left(
\int_{|x'|<\delta}
\frac{|x'|^2}
{\left(1+\lambda^p|x'|^{\frac{p}{p-1}}\right)^{n-1}}
\,dx'
\right).
\end{align*}
Hence
\begin{equation}\label{eq2.13}
R_2=
\begin{cases}
o\!\left(\dfrac{1}{\lambda^{(p-1)(n+1)}}\right),
& \text{if } p<\dfrac{n+1}{2},\\[6pt]
o\!\left(\dfrac{\log \lambda}{\lambda^{(p-1)(n+1)}}\right),
& \text{if } p=\dfrac{n+1}{2}.
\end{cases}
\end{equation}
And
\begin{equation*}
R_3
=
O\!\left(
\int_{|x'|<\delta}
\frac{\lambda^p|x'|^{\frac{2p}{p-1}-\theta+2}}
{\left(1+\lambda^p|x'|^{\frac{p}{p-1}}\right)^n}
\,dx'
\right).
\end{equation*}
For $\theta$ small enough, we obtain
\begin{equation}\label{eq2.14}
R_3
=
\begin{cases}
\displaystyle
o\!\left(\dfrac{1}{\lambda^{(p-1)(n+1)}}\right),
& \text{if } p<\dfrac{n+1}{2}, \\[8pt]
\displaystyle
o\!\left(\dfrac{\log \lambda}{\lambda^{(p-1)(n+1)}}\right),
& \text{if } p=\dfrac{n+1}{2}.
\end{cases}
\end{equation}
From \eqref{eq2.10}-\eqref{eq2.14} we find that
\begin{equation}\label{eq2.15}
J_1
=
\frac{c_1}{n-1}
\frac{\sum_{i=1}^{n-1}\gamma_i}
{\lambda^{(p-1)(n+1)}}
+
o\!\left(
\frac{1}{\lambda^{(p-1)(n+1)}}
\right)
\quad
\text{if } 1<p<\frac{n+1}{2}.
\end{equation}
and
\begin{equation}\label{eq2.155}
J_1=
\frac{\widehat{c}}{n-1}
\frac{\left(\sum_{i=1}^{n-1}\gamma_i\right)\log \lambda}{\lambda^{(p-1)(n+1)}}
+o\!\left(\frac{\log \lambda}{\lambda^{(p-1)(n+1)}}\right),
\qquad \text{if } p=\frac{n+1}{2}.
    \end{equation}
Now we estimate the second integral of \eqref{eq2.8}. Using the identity \eqref{eq.2.9} and the
same computation of $J_1$ we have
\[
J_2=
\int_{|x'|<\delta}
\frac{\varphi(x')}{\left(1+\lambda^p|x'|^{\frac{p}{p-1}}\right)^{n}}\,dx'
+O\!\left(
\int_{|x'|<\delta}
\frac{\lambda^p {\varphi(x')}^{\theta+1}|x'|^{\frac{p}{p-1}-\theta}}{\left(1+\lambda^p|x'|^{\frac{p}{p-1}}\right)^{n+1}}dx'
\right)
\]

\[
+O\!\left(
\int_{|x'|<\delta}
\frac{ \lambda^p|\varphi(x')|^{\frac{p}{p-1}-\theta+1}|x'|^\theta}
{\left(1+\lambda^p|x'|^{\frac{p}{p-1}}\right)^{n+1}}dx'
\right).
\]
Thus
\[
J_2=
\sum_{i=1}^{n-1}
\frac{\gamma_i}{\lambda^{(p-1)(n+1)}}
\int_{\mathbb{R}^{n-1}}
\frac{z_i^2}{\left(1+|z|^{\frac{p}{p-1}}\right)^n}\,dz
+o\!\left(\frac{1}{\lambda^{(p-1)(n+1)}}\right).
\]
Hence for any $1<p\le \frac{n+1}{2}$ it follows that
\begin{equation}\label{eq2.16}
J_2=
\frac{c_2}{n-1}
\frac{\sum_{i=1}^{n-1}\gamma_i}{\lambda^{(p-1)(n+1)}}
+o\!\left(\frac{1}{\lambda^{(p-1)(n+1)}}\right).
\end{equation}
From \eqref{eq2.15}, \eqref{eq2.155} and \eqref{eq2.16}, estimate \eqref{eq2.8} reduces to
\begin{equation}\label{eq2.17}
I_2=
\left(\frac{n-p}{p-1}\right)^p
\frac{c_1-c_2}{n-1}
\frac{\sum_{i=1}^{n-1}\gamma_i}{\lambda^{p-1}}
+o\!\left(\frac{1}{\lambda^{p-1}}\right),
\qquad \text{if } 1<p<\frac{n+1}{2},
\end{equation}
and
\begin{equation}\label{eq2.177}
I_2=
\left(\frac{n-p}{p-1}\right)^p
\frac{\widehat{c}}{n-1}
\frac{\left(\sum_{i=1}^{n-1}\gamma_i\right)\log \lambda}{\lambda^{p-1}}
+o\!\left(\frac{\log \lambda}{\lambda^{p-1}}\right),
\qquad \text{if } p=\frac{n+1}{2}.
    \end{equation}
The same computation works for $I_3$ and using the change of variables
$x_n \to -x_n$ we get
\begin{equation}\label{eq2.18}
I_3=-I_2.
    \end{equation}
Combining now \eqref{eq2.7}, \eqref{eq2.17},\eqref{eq2.177}, and \eqref{eq2.18} we obtain from \eqref{eq2.5}
\begin{equation}\label{eq2.19}
I_0=
\left(\frac{n-p}{p-1}\right)^p
\left[
\Sigma-\frac{2}{n-1}\frac{c_1-c_2}{\lambda^{p-1}}
\sum_{i=1}^{n-1}\gamma_i
+o\!\left(\frac{1}{\lambda^{p-1}}\right)
\right],
\qquad 1<p<\frac{n+1}{2},
    \end{equation}
and
\begin{equation}\label{eq2.20}
I_0=
\left(\frac{n-p}{p-1}\right)^p
\left[
\Sigma-
\frac{2}{n-1}
\frac{\left(\sum_{i=1}^{n-1}\gamma_i\right)\log \lambda}{\lambda^{p-1}}
+o\!\left(\frac{\log \lambda}{\lambda^{p-1}}\right)
\right],
\qquad p=\frac{n+1}{2}.
    \end{equation}
For the remainder term $R_0$ of \eqref{eq2.4}, we have
\[
R_0 \le
c\left(
\int_{|x'|>\frac{r}{2}}
|\nabla \psi|^p \delta_{(0,\lambda)}^p dx
+
\int_{|x'|>\frac{r}{2}}
|\theta|^p |\nabla\delta_{(0,\lambda)}|^p dx
\right).
\]
Observe that
\[
\int_{|x'|>\frac{r}{2}}
|\theta|^p|\nabla \delta_{(0,\lambda)}|^p dx
\le
\int_{|z|>\lambda^{p-1}\frac{r}{2}}
\frac{|z|^{\frac{p}{p-1}}}{(1+|z|^{\frac{p}{p-1}})^n}dz
=O\!\left(\frac{1}{\lambda^{n-p}}\right),
\]
and
\[
\begin{aligned}
\int_{|x'|>\frac{r}{2}}
|\nabla \psi|^p \delta_{(0,\lambda)}^p \, dx
&\le c \int_{\frac{r}{2}<|x|<r}
\frac{\lambda^{(p-1)(n-p)}}
{(1+\lambda^p|x|^{\frac{p}{p-1}})^{n-p}}\,dx  \\[6pt]
&\le \frac{c}{\lambda^{p(p-1)}}
\int_{\frac{r \lambda^{p-1}}{2}<|z|< r \lambda^{p-1}}
\frac{dz}{(1+|z|^{\frac{p}{p-1}})^{n-p}}
\le \frac{c}{\lambda^{\,n-p}} .
\end{aligned}
\]
It follows that
\begin{equation}\label{eq2.21}
R_0=
\begin{cases}
o\!\left(\dfrac{1}{\lambda^{p-1}}\right),
& 1<p<\dfrac{n+1}{2}, \\[6pt]
o\!\left(\dfrac{\log \lambda}{\lambda^{p-1}}\right),
& p=\dfrac{n+1}{2}.
\end{cases}
    \end{equation}
After recalling that $H(0)= \frac{2}{n-1}\sum_{i=1}^{n-1} \gamma_i$, the proof of Lemma \ref{lem2.2} follows from \eqref{eq2.4}, \eqref{eq2.19}, \eqref{eq2.20} and \eqref{eq2.21}. This completes the proof.
\end{proof}

\begin{lemma}\label{lem2.3}
For $1<p<n$, we have
\[
\int_{\Omega} U_{(a,\lambda)}^{p^*}\,dx
= \frac{1}{n}\left(\frac{n-p}{p-1}\right)\Sigma
- c_2\frac{H(a)}{\lambda^{p-1}}
+ o\!\left(\frac{1}{\lambda^{p-1}}\right),
\]
where $\Sigma$ and $c_2$ are defined in Lemma \ref{lem2.1}.
\end{lemma}
\begin{proof}
Using the definition of \(U_{(a,\lambda)}\) and the support properties of the cutoff function, we decompose the integral as follows
\begin{equation}\label{eq3.1}
\begin{aligned}
\int_{\Omega} U_{(a,\lambda)}^{p^*}\,dx
&=
\int_{\Omega\cap B(a,\frac{r}{2})} \delta_{(a,\lambda)}^{p^*}\,dx
+
O\!\left(\int_{|x-a|>\frac{r}{2}} \delta_{(a,\lambda)}^{p^*}\,dx\right) \\
&= I_0 + R_0.
\end{aligned}
\end{equation}
Setting
\(
z=\lambda^{{p-1}}(x-a),
\)
we get
\begin{equation}\label{eq.3.2}
R_0 = O\!\left(
\int_{|z|>\frac{r\lambda^{p-1}}{2}} \frac{dz}{(1+|z|^{\frac{p}{p-1}})^n}
\right)
= O\!\left(\frac{1}{\lambda^{\,n}}\right)
= o\!\left(\frac{1}{\lambda^{p-1}}\right).
    \end{equation}
Using the notations of the proof of Lemma \ref{lem2.1} we have
\begin{equation}\label{eq3.3}
\begin{aligned}
I_0
&=
\int_{B^+(0,r)} \delta_{(0,\lambda)}^{p^*}\,dx
-
\int_{\Sigma_2} \delta_{(0,\lambda)}^{p^*}\,dx
+
\int_{\Sigma_1} \delta_{(0,\lambda)}^{p^*}\,dx \\
&= I_1 - I_2 + I_3.
\end{aligned}
\end{equation}
Observe that
\[
I_1 =
\int_{\mathbb{R}^n_+} \delta_{(0,\lambda)}^{p^*}\,dx
+
O\!\left(\frac{1}{\lambda^{\,n}}\right).
\]
Using the fact that
\[
\begin{cases}
-\Delta_p \delta_{(0,\lambda)}
=
n\left(\frac{n-p}{p-1}\right)^{p-1}
\delta_{(0,\lambda)}^{p^*-1}
& \text{in } \mathbb{R}^n_+, \\[6pt]

\displaystyle
~~~~~~\quad \frac{\partial \delta_{(0,\lambda)}}{\partial x_n}
=
0
& \text{on } \partial\mathbb{R}^n_+ ,
\end{cases}
\]
we get
\begin{equation}\label{eq3.33}
\int_{\mathbb{R}^n_+} \delta_{(0,\lambda)}^{p^*}\,dx
=
\frac{1}{n}
\left(\frac{p-1}{n-p}\right)^{p-1}
\int_{\mathbb{R}^n_+} |\nabla \delta_{(0,\lambda)}|^p\,dx .
    \end{equation}
Using again \eqref{eq2.7}, we get
\[
\int_{\mathbb{R}^n_+} \delta_{(0,\lambda)}^{p^*}\,dx
=
\frac{1}{n}\frac{n-p}{p-1}\Sigma,
\]
and hence
\begin{equation}\label{eq3.4}
I_1
=
\frac{1}{n}\frac{n-p}{p-1}\Sigma
+
o\!\left(\frac{1}{\lambda^{p-1}}\right).
    \end{equation}
For the second integral of \eqref{eq3.3} we have
\[
I_2
=
\int_{\Sigma_2 \cap L_{\delta}} \delta_{(0,\lambda)}^{p^*}\,dx
+
O\!\left(\frac{1}{\lambda^{\,n}}\right)
=
\lambda^{\,n(p-1)}
\int_{\Sigma_2\cap L_\delta}
\frac{dx}
{(1+\lambda^p|x'|^{\frac{p}{p-1}})^n}
= \lambda^{n(p-1)} J_2.
\]
where $J_2$ is defined in \eqref{eq2.8}.
Thus, using an estimate
 \eqref{eq2.16}, we get
\begin{equation}\label{eq3.5}
I_2 = \frac{c_2}{(n-1)\lambda^{p-1}}\sum_{i=1}^{n-1}\gamma_i
+ o\!\left(\frac{1}{\lambda^{p-1}}\right).
\end{equation}
Moreover, we have
\begin{equation}\label{eq3.6}
I_3 = -I_2 .
    \end{equation}
Then the proof follows from \eqref{eq3.1}-\eqref{eq3.6}.

\end{proof}

\begin{lemma}\label{lem2.4}
Let $1<p\le \frac{n+1}{2}$. Then
\[
\int_\Omega \alpha(x)U_{(a,\lambda)}^p\,dx
=
\begin{cases}
o\!\left(\frac{1}{\lambda^{p-1}}\right),
& \text{if } p<\frac{n+1}{2},\\[6pt]

o\!\left(\frac{\log \lambda}{\lambda^{p-1}}\right),
& \text{if } p=\frac{n+1}{2}.
\end{cases}
\]
\end{lemma}
\begin{proof} We can write
\[
\int_\Omega \alpha(x)U_{(a,\lambda)}^p\,dx
=
\int_{\Omega\cap B(a,\frac{r}{2})}
\alpha(x)\delta_{(a,\lambda)}^p(x)\,dx
+
\int_{\Omega\cap B(a,\frac{r}{2})^c}
\alpha(x)U_{(a,\lambda)}^p(x)\,dx
\]
\[
= I_1 + R .
\]
Using the fact that $\alpha(x)\in L^\infty(\Omega)$ we have
\begin{equation}\label{eq4.1}
\begin{aligned}
|R|
&\le c
\int_{\Omega\cap B(a,\frac{r}{2})^c}
\frac{\lambda^{(p-1)(n-p)}}
{(1+\lambda^p|x-a|^{\frac{p}{p-1}})^{n-p}}\,dx\\
&\le
c
\int_{\Omega\cap B(a,\frac{r}{2})^c}
\frac{\lambda^{(p-1)(n-p)}}
{(1+\lambda^p  {{(\frac{r}{2})}}^{\frac{p}{p-1}})^{n-p}}\,dx
\le
\frac{c|\Omega|}{\lambda^{n-p}}.
      \end{aligned}
\end{equation}
In addition, by setting
\(
z=\lambda^{p-1}(x-a),
\) we have
\[
\begin{aligned}
|I_1|
&\le
\frac{c}{\lambda^{p(p-1)}}
\int_{|z|<\frac{r\lambda^{p-1}}{2}}
\frac{dz}{(1+|z|^{\frac{p}{p-1}})^{n-p}} \\
&\le
\frac{c}{\lambda^{p(p-1)}}
\left(
\int_0^A
\frac{r^{n-1}}{(1+r^{\frac{p}{p-1}})^{n-p}}\,dr
+
\int_A^{\frac{r \lambda^{p-1}}{2}}
\frac{r^{n-1}}{r^{\frac{p(n-p)}{p-1}}}\,dr
\right),
\end{aligned}
\]
where $A$ is a large positive constant.
Therefore
\begin{equation}\label{eq4.2}
I_1 =
O\!\left(\frac{1}{\lambda^{p(p-1)}}\right)
+
 {O\!\left(\frac{1}{\lambda^{n-p}}\right)}.
    \end{equation}
The proof follows from \eqref{eq4.1} and \eqref{eq4.2}.
    \end{proof}
\begin{lemma}\label{lem2.5}
For $1<p\le \frac{n+1}{2}$ we have
\[
\int_{\Gamma_1}\beta(x)U_{(a,\lambda)}^p\,d\sigma
=
\tilde c\,
\frac{\beta(a)}{\lambda^{(p-1)^2}}
\big(1+o(1)\big)
+
o\!\left(\frac{1}{\lambda^{n-p}}\right),
\]
where $\tilde c$ is a positive constant.
\end{lemma}
\begin{proof}
We write
\begin{equation}\label{eq.5.0}
\begin{aligned}
\int_{\Gamma_1}\beta(x)U_{(0,\lambda)}^p\,d\sigma
&=
\int_{\Gamma_1\cap B(a,\frac{r}{2})}
\beta(x)\delta_{(0,\lambda)}^p(x)\,d\sigma
\int_{\Gamma_1\cap B(a,\frac{r}{2})^c}
\beta(x)U_{(0,\lambda)}^p(x)\,d\sigma \\
&= I + R.
\end{aligned}
\end{equation}
Since $\beta(x)\in L^\infty(\Gamma_1)$, we have
\begin{equation}\label{eq.5.1}
|R|
\le
c \frac{|\Gamma_1|}{\lambda^{n-p}}.
    \end{equation}
For $\delta>0$ small, we define (assuming that $a=0$),
\[
L_\delta
=
\{(x',x_n)\in B\left(0,\frac{r}{2}\right):
|x'|<\delta\}.
\]
Therefore
\begin{equation}\label{eq.5.2}
\begin{aligned}
I
&=
\int_{\Gamma_1\cap L_\delta}
\beta(x)\delta_{(0,\lambda)}^p\,d\sigma
+
O\!\left(\frac{1}{\lambda^{n-p}}\right) \\
&=
\beta(0) \int_{\Gamma_1\cap L_\delta}
\delta_{(0,\lambda)}^p\,d\sigma
+
o\!\left(\int_{\Gamma_1\cap L_\delta}
\delta_{(0,\lambda)}^p\,d\sigma\right)
+
O\!\left(\frac{1}{\lambda^{n-p}}\right).
\end{aligned}
\end{equation}
as $\delta$ is small enough. Observe that
\begin{equation}\label{eq.5.3}
    \int_{\Gamma_1\cap L_\delta} \delta^p_{(0,\lambda)}\,d\sigma=\int_{|x'|<\delta}\frac{\lambda^{(p-1)(n-p)}(1+|\nabla\varphi(x')|^2)^{\frac{1}{2}}}{(1+\lambda^p|(x',\varphi(x'))|^{\frac{p}{p-1}})^{n-p}}dx'
\end{equation}
Therefore, by \eqref{eq.2.3} and \eqref{eq.2.9}, we write
\begin{equation}\label{eq.5.4}
(1+|\nabla\varphi(x')|^2)^{\frac12}
=
1+O(|x'|^2).
    \end{equation}
\[
|(x',\varphi(x'))|^{\frac{p}{p-1}}=|x'|^{\frac{p}{p-1}}+O\left(|\varphi(x')|^{\frac{p}{p-1}}+|\varphi(x')|^{\theta}|x'|^{\frac{p}{p-1}-\theta}+ |\varphi(x')|^{\frac{p}{p-1}-\theta}|x'|^{\theta}\right)
\]
for $\theta>0$ small enough. It follows that,
\begin{equation}\label{eq.5.5}
|(x',\varphi(x'))|^{\frac{p}{p-1}}
=
|x'|^{\frac{p}{p-1}}
+
O(|x'|^{\frac{p}{p-1}+\theta}).
    \end{equation}
Using \eqref{eq.5.4} and \eqref{eq.5.5}, estimate \eqref{eq.5.3} reduces to
\begin{equation*}
\begin{aligned}
\int_{\Gamma_1\cap L_\delta}
\delta_{(0,\lambda)}^p\,d\sigma
&=
\lambda^{(p-1)(n-p)}
\int_{|x'|<\delta}
\frac{\left(1+O(|x'|^2)\right)}
{\left(1+\lambda^p|x'|^{\frac{p}{p-1}}\right)^{n-p}}
\left(
1+
O\!\left(
\frac{\lambda^p|x'|^{\frac{p}{p-1}+\theta}}
{1+\lambda^p|x'|^{\frac{p}{p-1}}}
\right)
\right)
\,dx' \\
&=
\lambda^{(p-1)(n-p)}
\int_{|x'|<\delta}
\left[
\frac{1}
{\left(1+\lambda^p|x'|^{\frac{p}{p-1}}\right)^{n-p}}
+
O\!\left(
\frac{\lambda^p|x'|^{\frac{p}{p-1}+\theta}}
{\left(1+\lambda^p|x'|^{\frac{p}{p-1}}\right)^{n-p+1}}
\right)
\right] \\
&\quad \times
\left(1+O(|x'|^2)\right)
\,dx'.
\end{aligned}
\end{equation*}
Using fact that,
\[\frac{\lambda^p|x'|^{\frac{p}{p-1}+\theta}}{\left(1+\lambda^p|x'|^{\frac{p}{p-1}}\right)^{n-p+1}}= o\left( \frac{1}{\left(1+\lambda^p|x'|^{\frac{p}{p-1}}\right)^{n-p}}\right),\]
as $\delta$ is small enough, we get
\begin{equation}\label{eq.5.6}
\int_{\Gamma_1\cap L_\delta}
\delta_{(0,\lambda)}^p\,d\sigma
=
\lambda^{(p-1)(n-p)}
\int_{|x'|<\delta}
\frac{dx'}
{(1+\lambda^p|x'|^{\frac{p}{p-1}})^{n-p}} (1+o(1)) .
    \end{equation}
Setting
\[
z=\lambda^{p-1}x'.
\]
Then
\begin{equation}\label{eq.5.7}
\begin{aligned}
\int_{|x'|<\delta}
\frac{dx'}
{\left(1+\lambda^p|x'|^{\frac{p}{p-1}}\right)^{n-p}}
&=
\frac{1}{\lambda^{(p-1)(n-1)}}
\int_{|z|<\lambda^{p-1}\delta}
\frac{dz}
{\left(1+|z|^{\frac{p}{p-1}}\right)^{\,n-p}} \\
&=
\frac{\omega_{n-2}}{\lambda^{(p-1)(n-1)}}
\left(
\int_0^A
\frac{r^{n-2}\,dr}
{\left(1+r^{\frac{p}{p-1}}\right)^{n-p}}
+
O\left(\int_A^{\lambda^{p-1}\delta}
r^{n-2-\frac{p(n-p)}{p-1}}\,dr
\right)
\right) \\
&=
\frac{\omega_{n-2}}{\lambda^{(p-1)(n-1)}}
\left(
c
+
O\!\left(
\frac{1}{\lambda^{-p^2+p+n+1}}
\right)
\right).
\end{aligned}
\end{equation}
Thus, \eqref{eq.5.6} and \eqref{eq.5.7} yield
\begin{equation}\label{eq.5.8}
\int_{\Gamma_1}\delta_{(0,\lambda)}^{\,p}\,d\sigma
=
\frac{c\,\omega_{n-2}}{ {\lambda^{(p-1)^2}}}
+
O\!\left(\frac{1}{\lambda^{n-p}}\right)
+
o\!\left(\frac{1}{\lambda^{(p-1)^2}}\right).
\end{equation}
The proof follows from \eqref{eq.5.0}, \eqref{eq.5.1}, \eqref{eq.5.2} and \eqref{eq.5.8}. \qed

Consequently, for \(1<p\le \frac{n+1}{2}\), \(n\ge 2\), and
\[
\|U_{(a,\lambda)}\|^{p}
=
\int_\Omega
\left(|\nabla U_{(a,\lambda)}|^p+\alpha(x)U_{(a,\lambda)}^p\right)\,dx
+
\int_{\Gamma_1}\beta(x)U_{(a,\lambda)}^p\,d\sigma,
\]
the following holds.
\end{proof}

\begin{corollary}\label{cor2.6}
\begin{itemize}

\item[(i)] If \(1<p<2\) and \(n\ge 3\), or \(1<p\leq\frac32\) and \(n=2\), then
\[
\|U_{(a,\lambda)}\|^{p}
=
\left(\frac{n-p}{p-1}\right)^p
\Sigma
\left[
1+\frac{\bar c}{\Sigma}
\left(\frac{p-1}{n-p}\right)^p
\frac{\beta(a)}{\lambda^{(p-1)^2}}
+
o\!\left(\frac{1}{\lambda^{(p-1)^2}}\right)
\right].
\]

\item[(ii)] If \(p=2\) and \(n\ge 4\), then
\[
\|U_{(a,\lambda)}\|^{p}
=
\left(\frac{n-p}{p-1}\right)^p
\Sigma
\left[
1+
\left(
\frac{\bar c}{\Sigma}\left(\frac{p-1}{n-p}\right)^p\beta(a)
-
\frac{c_1-c_2}{\Sigma}H(a)
\right)\frac{1}{\lambda}
+
o\!\left(\frac{1}{\lambda}\right)
\right].
\]

\item[(iii)] If \(2<p<\frac{n+1}{2}\), \(n\ge 4\), then
\[
\|U_{(a,\lambda)}\|^{p}
=
\left(\frac{n-p}{p-1}\right)^p
\Sigma
\left[
1-\frac{c_1-c_2}{\Sigma}\frac{H(a)}{\lambda^{p-1}}
+
o\!\left(\frac{1}{\lambda^{p-1}}\right)
\right].
\]

\item[(iv)] If \(p=\frac{n+1}{2}\), \(n\ge 3\), then
\[
\|U_{(a,\lambda)}\|^{p}
=
\left(\frac{n-p}{p-1}\right)^p
\Sigma
\left[
1-\frac{\hat c}{\Sigma}
\frac{H(a)\log\lambda}{\lambda^{p-1}}
+
o\!\left(\frac{\log\lambda}{\lambda^{p-1}}\right)
\right].
\]

\end{itemize}
\end{corollary}
\begin{proof}
It follows from the estimates of Lemmas \ref{lem2.2}, \ref{lem2.4} and \ref{lem2.5}.
\end{proof}

Recall that from \eqref{eq1.3}, we have
\[
J(U_{(a,\lambda)})
=
\frac{\|U_{(a,\lambda)}\|^{p}}
{\left(\displaystyle\int_\Omega U_{(a,\lambda)}^{p^*}\,dx\right)^{\frac{p}{p^*}}}.
\]
The following results evaluate the level of
\(
J(U_{(a,\lambda)})
\)
for \(a\in \Gamma_1\) and \(\lambda\) large enough.

\begin{lemma}\label{lem2.7}
Let \(1<p<2\), \(n\ge 3\), \(\bigl(1<p<\frac32,\ \text{if } n=2\bigr)\).
Let \(a\in \Gamma_1\) such that \(\beta(a)<0\). Then
\[
J(U_{(a,\lambda)})<\frac{S}{2^{p/n}},
\qquad \text{for } \lambda \text{ large.}
\]
\end{lemma}

\begin{proof}
Observe first that by Lemma \ref{lem2.3}, we have
\begin{equation}\label{eq5.1}
\left(
\int_\Omega U_{(a,\lambda)}^{p^*}\,dx
\right)^{\frac{p}{p^*}}
=
\left(
\frac{\Sigma}{n}\left(\frac{n-p}{p-1}\right)
\right)^{\frac{n-p}{n}}
\left[
1-\frac{c_2}{\Sigma}\frac{(p-1 )H(a)}{\lambda^{p-1}}
+
o\!\left(\frac{1}{\lambda^{p-1}}\right)
\right].
    \end{equation}
This with the first assertion of Corollary \ref{cor2.6} yields,
\[
J(U_{(a,\lambda)})
=
n^{\frac{n-p}{n}}
\left(\frac{n-p}{p-1}\right)^{p+\frac{p}{n}-1}
\Sigma^{\frac{p}{n}}
\left[
1+
\frac{\bar c}{\Sigma}
\left(\frac{p-1}{n-p}\right)^p
\frac{\beta(a)}{\lambda^{(p-1)^2}}
+
o\!\left(\frac{1}{\lambda^{(p-1)^2}}\right)
\right],
\]
since
\[
\frac{1}{\lambda^{p-1}}
=
o\!\left(\frac{1}{\lambda^{(p-1)^2}}\right),
\qquad \text{for } 1<p<2.
\]
Using the fact that
\[
S=
\frac{\displaystyle\int_{\mathbb{R}^n}|\nabla \delta_{(a,\lambda)}|^p\,dx}
{\left(\displaystyle\int_{\mathbb{R}^n}\delta_{(a,\lambda)}^{p^*}\,dx\right)^{\frac{p}{p^*}}},
\]
and
\[
-\Delta_p \delta_{(a,\lambda)}
= n\left(\frac{n-p}{p-1}\right)^{p-1}
\delta_{(a,\lambda)}^{p^*-1}
\quad \text{in } \mathbb{R}^n,
\]
we get
\begin{equation}\label{eq5.2}
\Sigma
= n^{\frac{p-n}{n}}
\left(\frac{p-1}{n-p}\right)^{\frac{n(p-1)}{p}+1}
\left(\frac{S^{\frac{n}{p}}}{2}\right).
    \end{equation}
Therefore,
\[
J(U_{(a,\lambda)})
=
\frac{S}{2^{p/n}}
\left[
1+
\frac{\bar c}{\Sigma}
\left(\frac{p-1}{n-p}\right)^p
\frac{\beta(a)}{\lambda^{(p-1)^2}}
+
o\!\left(\frac{1}{\lambda^{(p-1)^2}}\right)
\right].
\]
For $\lambda$ large enough and $\beta(a)<0$, we get the desired estimate.
\end{proof}
\medskip

\textbf{Proof of Theorem 1.2.}
Under the assumption of Theorem \ref{thm1.2}, there exists at least $a\in\Gamma_1$ such  that $\beta(a)<0$. Using Lemmas \ref{lem2.1} and \ref{lem2.7}, the Sobolev quotient $Q_p(\Omega)$ is achieved.
Let $u\in {V}^{1,p}(\Omega)\setminus\{0\}$ be a minimizer of $Q_p(\Omega)$.
Using the fact that $|u|\in {V}^{1,p}(\Omega)\setminus\{0\}$ and $J(|u|)=J(u)$,
then $|u|$ is a minimizer of $Q_p(\Omega)$ and hence $|u|$ is a nontrivial solution of problem \eqref{eq1.1}.
  {Using the maximum principle, see \cite{Vazquez1984}, we derive that $|u|>0$ on $\Omega$}.

\medskip

\begin{lemma}\label{lem2.8}
Let $2< p<\frac{n+1}{2}$, $n\ge 4$ and let $a\in \Gamma_1$
be a point such that $\textbf{(g.c.)}$ condition is satisfied. Then, for $\lambda$ large enough, we have
\[
J(U_{(a,\lambda)})<\frac{S}{2^{p/n}}.
\]
\end{lemma}

\begin{proof}
Using the third assertion of Corollary \ref{cor2.6} and estimates \eqref{eq5.1} and \eqref{eq5.2}, we have
\begin{equation}\label{eq5.3}
J(U_{(a,\lambda)})
=
\frac{S}{2^{p/n}}
\left[
1-
\frac{(c_1-pc_2)}{\Sigma}\frac{H(a)}{\lambda^{p-1}}
+
o\!\left(\frac{1}{\lambda^{p-1}}\right)
\right].
    \end{equation}
where $c_1$ and $c_2$ are defined in Lemma \ref{lem2.2}. Using the fact that $H(a)>0$ and
\[
c_1-pc_2
=
\int_{\mathbb{R}^{n-1}}
\frac{|z|^2\left(|z|^{\frac{p}{p-1}}-(p-1)\right)}{\left(1+|z|^{\frac{p}{p-1}}\right)^n}
\,dz
>0,
\]
the result follows. \qed
\end{proof}

\medskip

\begin{lemma}\label{lem2.9}
Let $p=\frac{n+1}{2}$, $n\ge 3$ and let $a\in \Gamma_1$
be a point such that $\textbf{(g.c.)}$ condition is satisfied. Then, for $\lambda$ large enough, we have
\[
J(U_{(a,\lambda)})<\frac{S}{2^{p/n}}.
\]
\end{lemma}

\begin{proof}
The last assertion of Corollary \ref{cor2.6} and estimates \eqref{eq5.1} and \eqref{eq5.2} yield
\begin{equation}\label{eq5.4}
J(U_{(a,\lambda)})
=
\frac{S}{2^{p/n}}
\left[
1-
\frac{\hat c}{\Sigma}
\frac{H(a)\log\lambda}{\lambda^{p-1}}
+
o\!\left(\frac{\log\lambda}{\lambda^{p-1}}\right)
\right].
    \end{equation}
The proof follows since $H(a)>0$. \qed
\end{proof}

\medskip

\textbf{Proof of Theorem \ref{thm1.3}.}
The proof follows from Lemmas \ref{lem2.1}, \ref{lem2.8} and \ref{lem2.9}. \qed

\medskip

\textbf{Proof of Theorem \ref{thm1.4}.}
Under the assumptions of the theorem, the expansion of $J(U_{(a,\lambda)})$
reduces to the one of \eqref{eq5.3} or \eqref{eq5.4}. The result follows from Lemma \ref{lem2.1}. \qed

\section*{Declarations}

\noindent\textbf{Ethical Approval.}
Not applicable.

\medskip
\noindent\textbf{Competing interests.}
The authors declare that they have no competing interests.

\medskip
\noindent\textbf{Authors contributions.}
The authors contributed equally to this work.

\medskip
\noindent\textbf{Availability of data and materials.}
Data sharing is not applicable to this article as no new data were created or analyzed in this study.


\end{document}